\documentclass[]{amsart}

\usepackage{amsmath, amsthm, amscd, amsfonts, amssymb, graphicx, color}
\usepackage{color}
\usepackage{url}
\usepackage{xcolor}
\newtheorem{theorem}{Theorem}[section]
\newtheorem{lemma}[theorem]{Lemma}
\newtheorem{corollary}[theorem]{Corollary}
\newtheorem{proposition}[theorem]{Proposition}

\newtheorem{definition}[theorem]{Definition}

\newtheorem{remark}{Remark}

\makeatletter
\@namedef{subjclassname@2020}{%
	\textup{2020} Mathematics Subject Classification}
\makeatother

\newcommand{\Z}{\mathbb{Z}}
\begin{document}
	
	\title[]{Singularly cospectral circulant graphs}
	
	\author[C.M. Conde]{Cristian M. Conde${}^{1,2}$}
	\author[E. Dratman]{Ezequiel Dratman${}^{1,2}$}
	\author[L.N. Grippo]{Luciano N. Grippo${}^{1,2}$}
	\author[M. Privitelli]{Melina Privitelli${}^{2,3}$}
	
	\address{${}^{1}$Universidad Nacional de General Sarmiento. Instituto de Ciencias, Argentina\\}
	\address{${}^{2}$Consejo Nacional de Investigaciones Cient\'ificas y Tecnicas, Argentina}
    \address{${}^{3}$Universidad Nacional de Hurlingham. Instituto de tecnolog\'ia e ingenier\'ia, Argentina}

	\email{cconde@campus.ungs.edu.ar}
	\email{edratman@campus.ungs.edu.ar}
	\email{lgrippo@campus.ungs.edu.ar}
	\email{mprivite@campus.ungs.edu.ar}

	\keywords{Circulant graphs, cospectral graphs, singularly cospectral graphs}
	\subjclass[2020]{05C25, 05C50}
	
	\date{}

\begin{abstract}
		Two graphs having the same spectrum are said to be cospectral. Two graphs such that the absolute values of their nonzero eigenvalues coincide are singularly cospectral graphs. Cospectrality implies singular cospectrality, but the converse may be false. In this paper, we present sufficient conditions for two circulant graphs, with an even number of vertices, to be noncospectral singularly cospectral graphs. In this analysis, we study when a pair of these graphs have the same or distinct inertia. In addition, we show that two singularly cospectral circulant graphs with an odd prime number of vertices are isomorphic.
\end{abstract}

\maketitle

\section{Introduction}\label{s1}

	Let $\mathbb{Z}_n$ denote the additive group of integers modulo $n$, and let $S$ be a subset of $\mathbb{Z}_n\setminus\{0\}$ closed with respect to taking the additive inverse; i.e., $S=-S$. A graph $G=G(\Z_n,S)$ is \emph{circulant of order $n$} if it can be constructed as follows. The vertices of $G$ are the elements of $\Z_n$, and $ij$ is an edge of $G$ if and only if $j-i\in S$. Notice that the adjacency matrix $A_G$, of a circulant graph $G$, is a circulant matrix for some order of its vertices. Because of their many applications,  circulant graphs have been widely studied (see~\cite{Mankhova2012} and the references therein). In connection with the spectra of circulant graphs, in 2006,  So characterized those circulant graphs with integral eigenvalues~\cite{SO2006153} and conjectured that integral circulant graphs are isomorphic if and only if they are cospectral. Sander and Sander proved this conjecture in 2015~\cite{SanderandSander2015}.  Recently, the first three authors presented a necessary and sufficient condition for two graphs to be singularly cospectral~\cite{CondeLAMA} (i.e. the absolute values of their nonzero eigenvalues coincide). In the same article, they presented constructions of families of pairs of nonisomorphic singularly cospectral graphs and graph classes where singular cospectrality implies almost cospectrality. 
	
	The energy of a graph was defined by Gutman in 1978 as the sum of its singular values counted with their multiplicity~\cite{Gutman2001}. In 2005, Stevanovi\'c and Stankovi\'c proved that almost all circulant graphs are hyperenergetic~\cite{STEVANOVIC2005}, i.e, that their energies are greater than two times the number of vertices minus one.	 Blackburn and  Shparlinski, in 2008, found a lower and upper bound for the average energy of circulant graphs~\cite{Blackburn2008}. In this article, we focus on finding families of pairs of noncospectral singularly cospectral circulant graphs, with an even number of vertices. For an odd prime number of vertices, we prove that there is no pair of nonisomorphic singularly cospectral circulant graphs. 
	
	The remainder of the article is organized as follows. In section \ref{s2}, we introduce some definitions and notations, and we also present some preliminary results. Section \ref{sec: distinct-same-inertia} is devoted to giving sufficient conditions for two circulant graphs, with an even number of vertices, to have distinct (the same) inertia. In section \ref{sec: same-inertia}, we present sufficient conditions for two circulant graphs having or not having the same inertia, with an even number of vertices, to be noncospectral singularly cospectral graphs. Finally, in section \ref{sec: prime-case} we show that for two circulant graphs with an odd prime number of vertices, to be singularly cospectral implies to be isomorphic.
	
\section{Preliminaries}\label{s2}

All graphs mentioned in this article are finite and have no loops nor multiple edges. Let $G$ be a graph. We use $V(G)$ and $E(G)$ to denote the set of $n=|V(G)|$ vertices and the set of $m=|E(G)|$ edges of $G$, respectively. 
The adjacency matrix associated with the graph $G$ is defined by
$A_G= (a_{ij})_{n\times n}$, where $a_{ij} = 1$ if  the vertex $v_i$ is adjacent to the vertex
$v_j$ and $a_{ij} = 0$ otherwise.  We often denote $A_G$ simply by $A.$ We use $d_G(v)$ to denote the degree of $v$ in $G$ (the number of edges incident to $v$), or $d(v)$ provided the context is clear.  A $d$-regular graph is a graph such that $d_G(v)=d$ for every $v\in V(G)$.

The inertia of $G$ is the ordered triple $(\pi(G), \nu(G), \delta(G))$,
where $\pi(G)$ is the number of positive eigenvalues of $A_G$, $\nu(G)$ is the number of
negative eigenvalues of $A_G$, and $\delta(G)$ is the multiplicity of $0$ as an eigenvalue of $A_G$.

Let $a_h$, $1\leq h \leq s$, be integers such that $1\leq a_1<a_2 \cdots <a_s\leq \frac{n-1}{2}$.  If $G(\mathbb{Z}_n,S)$ is a circulant graph with $S:=\{a_1,\ldots,a_s,n-a_s,\ldots,n-a_1\}$, then it is well-known that its eigenvalues satisfy 
	\begin{equation} \label{formula autovalores}
	\lambda_j^{(s)}=\sum_{h=0}^{n-1} \alpha_{0h}\omega^{hj}=\sum_{h=1}^s(\omega^{ja_h}+\omega^{(n-a_h)j}),\,\,
	\end{equation}
for any $0\leq j \leq n-1$ where  $\omega=\exp(\frac{2\pi i}{n})$  and $(\alpha_{00}, \alpha_{01}, \dots, \alpha_{0(n-1)})$ is the first row of an adjacency circulant  matrix $A_G$ (see, e.g.,~\cite{biggs93}).

The next lemma  exhibits an infinite family of pairs
of singularly cospectral circulant graphs with an even number of vertices.

\begin{lemma}\label{lem: SC n=2k}
	Let $G_1(\Z_n,S_1)$ and $G_2(\Z_n,S_2)$ be two circulant graphs with $n=2k$. If $S_1=\{a_1,a_2,\ldots,a_s,n-a_s, \ldots,n-a_2,n-a_1\}$ with $1\le a_1<a_2<\cdots<a_s\le k-1$ and $S_2=\{b_1,b_2,\ldots,b_s,n-b_s, \ldots, n-b_2, n-b_1\}$ with $b_i=k-a_{s-i+1}$ for each $1\le i\le s$, then $G_1(\Z_n,S_1)$ and $G_2(\Z_n,S_2)$ are singularly cospectral circulant graphs.
\end{lemma}

\begin{proof}
	On  one hand, the eigenvalues of $G_1$ are
	\begin{equation*}
	\begin{split}
	\lambda_j^{(s)}&=\sum_{h=1}^s \omega^{a_hj}+\sum_{h=1}^s \omega^{(n-a_h)j}=\sum_{h=1}^s (\omega^{a_hj}+\omega^{-a_hj}),
	\end{split}
	\end{equation*}
	for every $0\le j\le n-1$. On the other hand, the eigenvalues of $G_2$ are
	\begin{equation}\label{eq:beta igual menos lambda}
	\begin{split}
	\beta_j^{(s)}&=\sum_{h=1}^s \omega^{b_hj}+\sum_{h=1}^s \omega^{(n-b_h)j}\\	
	&=\omega^{kj}\sum_{h=1}^s(\omega^{a_hj}+\omega^{-a_hj})=(-1)^j\sum_{h=1}^s(\omega^{a_hj}+\omega^{-a_hj})=(-1)^j\lambda_j^{(s)},
	\end{split}
	\end{equation}
	for every $0\le j\le n-1$. Therefore, we conclude that $G_1$ and $G_2$ are singularly cospectral.
\end{proof}

The eigenvalues of circulants graphs can be expressed in terms of the Chebyshev polynomials of the first and second kind. We introduce the definitions and different results on the Chebyshev polynomials which  play an
important role in our analysis.
\begin{definition}  Let $n$ be a non-negative integer. The Chebyshev polynomials of first kind $T_n$ are defined as $T_n(\cos(\theta))=\cos(n\theta)$.
	Similarly, the Chebyshev polynomials of the second kind $U_n$ are given by $U_n(\cos(\theta))\sin(\theta)=\sin((n+1)\theta)$.
\end{definition}
The results about Chebyshev polynomials of the first and second kind collected in the following
lemma are well-known (see ~\cite{MasonHandscomb}).
\begin{lemma}\label{lemma: prop polinomios de Chebyshev}
	The polynomials $T_n(X)$ and $U_n(X)$ satisfy the following properties, where 
	$X=\cos(\theta)$.
	\begin{enumerate}
		\item The polynomials $T_n$ can be defined through the following recurrence relation: $T_0(X)=1$, $T_1(X)=X$ and  $T_{n+1}=2XT_n-T_{n-1}$.
		\item If $k$ is an odd number, then $T_{\frac{k-1}{2}-l}(\cos(\frac{(2j+1)\pi}{k}))+T_{\frac{k+1}{2}+l}(\cos(\frac{(2j+1)\pi}{k}))=0$ for $l\geq 0$ and $j\in \mathbb{Z}.$
		\item If $k$ is an even number, then $T_{\frac{k}{2}-l}(\cos(\frac{(2j+1)\pi}{k}))+T_{\frac{k}{2}+l}(\cos(\frac{(2j+1)\pi}{k}))=0$ for $l\geq 0$ and $j\in \mathbb{Z}.$
		\item The polynomials $U_n(X)$ can be defined by the following recurrence relations: $U_0(X)=1$, $U_1(X)=2X$ and $U_{n+1}(X)=2XU_n(X)-U_{n-1}(X)$.
		\item $U_n(X)=\frac{\sin((n+1)\theta)}{\sin(\theta)}$.
		\item  $T_n(X)=\frac{1}{2}(U_n(X)-U_{n-2}(X))$.
		\item $U_n^{'}(X)=\frac{(n+1)T_{n+1}(X)-XU_n(X)}{X^2-1}. $
	\end{enumerate}
\end{lemma}

We now provide  expressions for the eigenvalues of $G(\mathbb{Z}_n,S)$
by means of Chebyshev polynomials of the first and second kind.
We may immediately deduce it  from Lemma \ref{lemma: prop polinomios de Chebyshev}. 

\begin{lemma} \label{lemma: autovalores dados por Chebyshev}
	Let $G(\Z_n,S)$ be a circulant graph. If $S=\{a_1,\ldots,a_s,n-a_s,\ldots,n-a_1\}$ with $1\le a_1<a_2<\cdots<a_s\le \frac{n-1}{2}$, then
	\begin{enumerate}
		\item \label{lambdaChebyshev} Let  $X_j=\cos(j\frac{2\pi}{n})$ for $0\le j \le n-1$, then  \begin{eqnarray*}
		\lambda_{j}^{(s)}=2\sum_{h=1}^s T_{a_h}(X_j).
		\end{eqnarray*}
		If $a_{i+1}=a_i+1$ for $1\le i\le s-1$, then
		\begin{eqnarray*}
		\lambda_{j}^{(s)}&=& U_{a_s}(X_j)+U_{a_{s}-1}(X_j)-U_{a_1 - 1}(X_j)-U_{a_{1}-2}(X_j) \label{eq: lambda dados por U}\\
		&=& U_{2a_s}(Y_j)-U_{2(a_1-1)}(Y_j), \nonumber
		\end{eqnarray*} 
		with $Y_j=\cos(j\frac{\pi}{n})$ and $U_{-1}(X_j):=0$.
		
		\item\label{beta igual lambda} Let $\hat{G}(\Z_n,\hat{S})$ be a circulant graph with $n=2k$. If $\hat{S}=\{b_1,b_2,\ldots,b_s,n-b_s, \ldots, n-b_2, n-b_1\}$ with $b_i=k-a_{s-i+1}$ for each $1\le i\le s$, then $\beta_j^{(s)}=(-1)^j\lambda_{j}^{(s)}$ for $0\leq j \leq n-1$, where $\beta_j^{(s)}$ denote the eigenvalues of $\hat{G}$. In particular, if $j $ is an even number we have that 
		$\beta_j^{(s)}=\lambda_{j}^{(s)}.$
		
		\item Under the assumption of item \eqref{beta igual lambda}, we have $\lambda_0^{(s)} =\beta_0^{(s)}=2s.$

		\item\label{simetria lambda} $\lambda_j^{(s)}=\lambda_{n-j}^{(s)}$, $1\leq j\leq \frac{n-1}{2}$.
		\item\label{lambdak} If $n=2k$, then $\lambda_k^{(s)}=-2(s_o - s_e)$ where $s_o=|\{a\in\{a_1,\dots,a_s\}:\text{ $a$ is odd}\}|$ and $s_e=|\{a\in\{a_1,\dots,a_s\}:\text{ $a$ is even}\}|$.
		\item\label{lambdacreciente} If $n=2k$ then $0<\lambda_1^{(2)}< \lambda_1^{(3)}< \cdots < \lambda_1^{(\lfloor\frac{k-1}{2}\rfloor)}$ where $\lambda_1^{(j)}$ is the  eigenvalue given by \eqref{lambdaChebyshev} with  $S=\{1,2,\ldots,j,n-j,\ldots,n-1\}.$
	\end{enumerate}
\end{lemma}

\begin{proof}

	We first prove the item (1). Note  that, by identity \eqref{formula autovalores},
	\begin{eqnarray}
	\lambda_j^{(s)}&=&\sum_{h=1}^s\exp{\frac{2\pi j a_h}{n}i} + \sum_{h=1}^s\exp{\frac{2\pi j(n-a_h)}{n}i}=2\sum_{h=1}^s\frac{\exp{\frac{2\pi j a_h}{n}i}+\exp{\frac{-2\pi j a_h}{n}i}}{2}\nonumber\\&=&2\sum_{h=1}^s T_{a_h}(X_j).\nonumber \
	\end{eqnarray} 
	Suppose that $a_{i+1}=a_i+1$ for $1\le i\le s-1$, from Lemma \ref{lemma: prop polinomios de Chebyshev}, item (6),  we have
	\begin{align*}
	\lambda_j^{(s)}&=2\sum_{h=1}^s T_{a_h}(X_j)=
	\begin{cases}
	\sum_{h=2}^s(U_{h}(X_j)-U_{h-2}(X_j)) + 2T_1(X_j) & \text{ if $a_1=1$}, \\
	\sum_{h=1}^s(U_{a_h}(X_j)-U_{a_h-2}(X_j)) & \text{ if $a_1 \ge 2$},
	\end{cases}\\
	&= \begin{cases}
	U_s(X_j)+U_{s-1}(X_j)-U_1(X_j)-U_0(X_j) + 2X_j & \text{ if $a_1=1$}, \\
	U_{a_s}(X_j)+U_{a_s-1}(X_j)-U_{a_2-2}(X_j)-U_{a_1-2}(X_j) & \text{ if $a_1 \ge 2$},
	\end{cases}\\
	&= \begin{cases}
	U_s(X_j)+U_{s-1}(X_j)-2X_j-1 + 2X_j & \text{ if $a_1=1$}, \\
	U_{a_s}(X_j)+U_{a_s-1}(X_j)-U_{a_1-1}(X_j)-U_{a_1-2}(X_j) & \text{ if $a_1 \ge 2$},
	\end{cases}\\
	\end{align*}
	in conclusion, we have that 
	\begin{equation}\label{formula: autovalores en fc de chebysshev 2 clase}
	\lambda_j^{(s)}=U_{a_s}(X_j)+U_{a_s-1}(X_j)-U_{a_1-1}(X_j)-U_{a_1-2}(X_j),
	\end{equation}with $U_{-1}(X_j):=0$.
	
	On the other hand,  observe that for $a \in \mathbb{Z}_{\ge 0}$, the following identity holds:
	\begin{align*}
	\frac{\sin((a+1)j\frac{2\pi}{n})+\sin(a j\frac{2\pi}{n})}{\sin(j\frac{2\pi}{n})} &= \frac{2\sin\left(\frac{(a+1)j\frac{2\pi}{n}+aj\frac{2\pi}{n}}{2}\right)\cdot \cos\left(\frac{(a+1)j\frac{2\pi}{n}-aj\frac{2\pi}{n}}{2}\right)}{\sin(j\frac{2\pi}{n})}\\
	&=\frac{2\sin\left(\frac{(2a+1)j\pi}{n}\right)\cdot \cos(\frac{j\pi}{n})}{\sin(j\frac{2\pi}{n})}\\
	&= \frac{2\sin\left(\frac{(2a+1)j\pi}{n}\right)\cdot \cos(\frac{j\pi}{n})}{2\sin(j\frac{\pi}{n}) \cdot \cos(j\frac{\pi}{n})}\\
	&= \frac{\sin\left(\frac{(2a+1)j\pi}{n}\right)}{\sin(j\frac{\pi}{n})} = U_{2a}\left(\cos\left(j\frac{\pi}{n}\right)\right),
	\end{align*}
	Then, taking into account the above formula and identity \eqref{formula: autovalores en fc de chebysshev 2 clase}, we conclude that 
	$\lambda_j^{(s)}=U_{2a_s}(Y_j)-U_{2(a_1-1)}(Y_j)$ with  $Y_j=\cos(j\frac{\pi}{n})$.
	
	The statements $(2), (3),  (4)$ and $(5)$ follow directly from the formula \eqref{formula autovalores}. 
	Finally, \eqref{lambdacreciente} follows from  the fact that  $\cos(h\frac{\pi}{k})>0$ for all $1\leq h\leq s$ and $2\leq s\leq \frac{k-1}{2}$.
\end{proof}

\section{The inertia of circulant graphs }\label{sec: distinct-same-inertia}

In this section we provide sufficient conditions under which two circulant graphs with an even number of vertices have different or the same inertia.

\begin{theorem}\label{thm: particular case}
	Let $G_1(\Z_n,S_1)$ and $G_2(\Z_n,S_2)$ be two circulant graphs with $n=2k$ and $k\ge 6$, where  $S_1=\{1,2,2k-2,2k-1\}$ and $S_2=\{k-2,k-1,k+1,k+2\}$. Then $G_1(\Z_n,S_1)$ and $G_2(\Z_n,S_2)$ have  different inertias.
\end{theorem}
\begin{proof}
	Let $\lambda_j$ and $\beta_j$ be the eigenvalues of $G_1$ and $G_2$, respectively, with $ 0 \le j \le n-1$. From Lemma~\ref{lem: SC n=2k}, $G_1$ and $G_2$ are singularly cospectral. Besides, $\lambda_{2r}=\beta_{2r}$ for each $0\le r\le k-1$ and $\lambda_{j}=2f(\frac{j\pi}{k})$ for each $0\le j\le n-1$, where $f(x)=2\cos^2(x)+\cos(x)-1$. In order to prove that they have distinct inertias we will prove that $P_k< N_k$ for every $k\ge 6$, where $P_k=|\{r\in\{1,\ldots,k\}:\, \lambda_{2r-1}>0\}|$ and $N_k=|\{r\in\{1,\ldots,k\}:\, \lambda_{2r-1}<0\}|$. 
	\begin{equation}\label{eq: P_k}
	P_k=\left|\left\{r\in\{1,\ldots,k\}:\; 0<\frac{2r-1}{k}<\frac1 3\right\}\right|+\left|\left\{r\in\{1,\ldots,k\}:\;  \frac 5 3<\frac{2r-1} {k}<2\right\}\right|
	\end{equation}
	and
	\begin{equation}\label{eq: N_k}
	N_k=\left|\left\{r\in\{1,\ldots,k\}:\; \frac 1 3<\frac{2r-1}{k}<1\right\}\right|+\left|\left\{r\in\{1,\ldots,k\}:\;  1<\frac{2r-1} {k}<\frac 5 3\right\}\right|.
	\end{equation}
	From~\eqref{eq: P_k} and~\eqref{eq: N_k} we obtain that
	\begin{equation}\label{eq: P_k formula}
	P_k=
	\begin{cases}
	k+\lfloor\frac{k+3} 6\rfloor-\lfloor\frac{5k+3} 6\rfloor-1& \mbox{if } k\equiv 3\ (\textrm{mod}\ 6),\\
	k+\lfloor\frac{k+3} 6\rfloor-\lfloor\frac{5k+3} 6\rfloor& \mbox{if } k\not\equiv 3\ (\textrm{mod}\ 6),
	\end{cases}
	\end{equation}
	and
	\begin{equation}\label{eq: N_k formula}
N_k=
	\begin{cases}
	\lfloor\frac{5k+3} 6\rfloor-\lfloor\frac{k+3} 6\rfloor-2& \mbox{if } k\equiv 3\ (\textrm{mod}\ 6)\\
	\lfloor\frac{5k+3} 6\rfloor-\lfloor\frac{k+3} 6\rfloor-1& \mbox{if } k\equiv 1\ (\textrm{mod}\ 2) \mbox{ and }k\not\equiv 3\ (\textrm{mod}\ 6)\\
	\lfloor\frac{5k+3} 6\rfloor-\lfloor\frac{k+3} 6\rfloor& \mbox{if } k\not\equiv 1\ (\textrm{mod}\ 2). 
	\end{cases}
		\end{equation}
	It is easy to verify by using~\eqref{eq: N_k formula} and~\eqref{eq: P_k formula} that $P_k<N_k$ for every $6\le k\le 12$. Let us consider $k\ge 13$. On the one hand, 
	\begin{equation*}
	P_k \le k+\lfloor\frac{k+3} 6\rfloor-\lfloor\frac{5k+3} 6\rfloor\le \frac{k+3} 3.
	\end{equation*}
	On the other hand,
	\begin{equation*}
	N_k \ge \lfloor\frac{5k+3} 6\rfloor-\lfloor\frac{k+3} 6\rfloor-2\ge \frac{2k-9}3.
	\end{equation*}
	Clearly, as $k+3 < 2k-9$, therefore, $P_k<N_k$.
\end{proof}

\begin{corollary}\label{cor:distinctinertia}
 $G_1(\Z_n,S_1)$ and $G_2(\Z_n,S_2)$  are non cospectral singularly cospectral (NCSC) circulant graphs.
\end{corollary}

Let $s$, $k$, $n$ be integers such that $k\geq 6$, $n=2k$ and $2\leq s \leq k-3$, and let 
	$S_1=\{1,\ldots,s,n-s,\ldots,n-2,n-1\}$ and $S_2=\{k-s,\ldots,k-1,k+1,\ldots,k+s\}$. The circulant graphs $G_1:=G_1(\mathbb{Z}_n,S_1)$ and $ G_2:=G_2(\mathbb{Z}_n,S_2)$ might have the same inertia (see Theorem~\ref{thm: same inertia}).  Theorem~\ref{lambdas} generalizes Corollary~\ref{cor:distinctinertia} by considering a pair of circulant graphs with the same inertia. Nevertheless,   the used technique   is completely different and its proof needs more sophisticated arguments.

We now provide sufficient conditions for pairs of circulant graphs as in Lemma ~\ref{lem: SC n=2k} to have the same inertia. 
\begin{theorem}\label{thm: same inertia}
		Let  $S_1=\{1,\ldots,s, n-s, \ldots, n-1\}$ and $S_2=\{k-s,\ldots,k-1, k+1, \ldots, k+s\}$. Then, the circulant graphs $G_1(\Z_n,S_1)$ and $G_2(\Z_n,S_2)$,  with $n=2k,$ $k=4\alpha +9$, $s=2\alpha +4$, and $ \alpha \geq 0$,  have the same  inertia. 
		\begin{proof}			
			 Let $\lambda_j^{(s)}$ and $\beta_j^{(s)}$ with $0\leq j \leq n-1$, be the eigenvalues of $G_1$ and $G_2$ respectively, then we have by Lemma \ref{lemma: autovalores dados por Chebyshev} 
			 \begin{enumerate}
			 	\item $\lambda_0^{(s)} =\beta_0^{(s)}=2s$, $\lambda_k=-2(s_o - s_e)=\beta_k=0$ and  $\beta_j^{(s)}=(-1)^j\lambda_{j}^{(s)}$ for any  $0\leq j \leq n-1$. 
			 	\item $\lambda_j^{(s)}=\lambda_{n-j}^{(s)}$ if  $1\leq j\leq \frac{n-1}{2}$.
			 \end{enumerate}  In particular, if $j $ is an even number we have 
			 $\beta_j^{(s)}=\lambda_{j}^{(s)}$, so to obtain the inertia of $G_1$ and $G_2$ respectively, we only need to analyze the sign of the eigenvalues $\lambda_j^{(s)}$ with $j$ an  odd  number such that $1\leq j \leq 4\alpha+7.$  We denote by $$J=\{j\in \mathbb{N}: j \:\textrm{is an  odd number and}\: 1\leq j\leq 4\alpha+7\}.$$ 
			 Note that, as consequence of  Lemma \ref{lemma: autovalores dados por Chebyshev} item \eqref{lambdaChebyshev}, we obtain 
			 	\begin{align*}
			 \lambda_{j}^{(s)}&=
			 U_{2s}\left(\cos\left(j\frac{\pi}{8\alpha+18}\right)\right)-1
			 \\&= \frac{\sin\left((4\alpha+8+1)\left(j\frac{\pi}{8\alpha+18}\right)\right)}{\sin\left( j\frac{\pi}{8\alpha+18}\right)}-1\\&=
			 \frac{\sin\left(j\frac{\pi}{2}\right)}{\sin\left(j\frac{\pi}{2}\frac{1}{4\alpha+9}\right)}-1=\frac{\sin\left(j\frac{\pi}{2}\right)-\sin\left(j\frac{\pi}{2}\frac{1}{4\alpha+9}\right)}{\sin\left(j\frac{\pi}{2}\frac{1}{4\alpha+9}\right)}\\&=
			 \frac{2\cos\left(j\frac{\pi}{4}\left(1+\frac{1}{4\alpha+9}\right)\right)\sin\left(j\frac{\pi}{4}\left(1-\frac{1}{4\alpha+9}\right)\right)}{\sin\left(j\frac{\pi}{2}\frac{1}{4\alpha+9}\right)}
			 \end{align*}
			 for any $0\leq j\leq n-1$.
			 
			 If $j\in J$,   then  $0<j\frac{\pi}{2}\frac{1}{4\alpha+9}<\frac{\pi}{2}$, $\sin\left(j\frac{\pi}{2}\frac{1}{4\alpha+9}\right)>0$ and from now on for the sake of simplicity, we denote for any 
			 \begin{align*}
			 \gamma_j=j\frac{\pi}{4}\left(1+\frac{1}{4\alpha+9}\right) \quad \textrm{and}\quad \delta_j=j\frac{\pi}{4}\left(1-\frac{1}{4\alpha+9}\right).
			 \end{align*}
			 Thus, \begin{align*}
		\lambda_{j}^{(s)}= \frac{2\cos(\gamma_j)\sin(\delta_j)}{\sin\left(j\frac{\pi}{2}\frac{1}{4\alpha+9}\right)}.
			 \end{align*}
			 As any $j\in J$ satisfies the inequality $1\leq j < 4\alpha +9$, then we conclude that 
			 \begin{align*}
			\left(\frac{j-1}{2}\right)\frac{\pi}{2}< \gamma_j, \delta_j < \left(\frac{j+1}{2}\right)\frac{\pi}{2}.
			 \end{align*}
			 From the oddness of $j$ we have to consider only two possible cases:
			 \begin{enumerate}
			 	\item Case 1: If $\frac{j-1}{2}$ is an even number, then $2\cos(\gamma_j)\sin(\delta_j)>0$. Combining this with the fact that $\sin\left(j\frac{\pi}{2}\frac{1}{4\alpha+9}\right)>0$, we conclude that $\lambda_{j}^{(s)}>0$.
			 	\item Case 2: If $\frac{j-1}{2}$ is an odd  number, then with a reasoning similar to that done in the previous case we conclude that $\lambda_{j}^{(s)}<0$.\end{enumerate}
		 	In conclusion,  for any odd number $j$ such that $1\leq j \leq 4\alpha +7$  we get 
		 	\begin{equation*}
		 	\begin{cases}
		 \lambda_{j}^{(s)}>0& \mbox{if } \frac{j-1}{2} \mbox{ is an even number} ,\\
		 	\lambda_{j}^{(s)}<0& \mbox{if } \frac{j-1}{2}\mbox{ is an odd number}.
		 	\end{cases}
		 	\end{equation*}
		  Thus the cardinality of the set $J$ is an even number and this allows us to conclude that 
		 	\begin{align*}
		 	\left| \Big\{\lambda_{j}^{(s)}: \lambda_{j}^{(s)}>0\:\textrm{and}\: j\in J\Big\}\right|=\left| \Big\{\lambda_{j}^{(s)}: \lambda_{j}^{(s)}<0 \:\textrm{and}\: j\in J\Big\}\right|.
		 	\end{align*}
		 	Therefore, $G_1(\Z_n,S_1)$ and $G_2(\Z_n,S_2)$ have the same inertia. 
				\end{proof}
\end{theorem}

\section{NCSC circulant graphs }\label{sec: same-inertia}

This section is devoted to providing sufficient conditions under which two circulant graphs with an even number of vertices are NCSC graphs. In order to obtain the main result of this section, we firstly need to establish some properties of the eigenvalues of circulant graphs.

\subsection{Some technical results} 
 With the aid of Lemma \ref{lemma: prop polinomios de Chebyshev}, we reach the following  result for certain eigenvalues. 
 
 \begin{proposition} \label{lemma: autovalores (s)=(k-s-1)} Let $G_1(\Z_n,S_1)$ and $G_2(\Z_n,S_2)$ be  circulant graphs with $n=2k$, $S_1=\{1,2, \ldots,s,n-s,\ldots,n-1\}$,  $S_2=\{1,2, \ldots,k-s-1,n-(k-s-1),\ldots,n-1\}$, $2\leq s \leq \frac{k-1}{2}$ and $0\leq j \leq  \frac{k-1}{2}$. Then,  
 	$\lambda_{2j+1}^{(s)}=\beta_{2j+1}^{(k-s-1)}$, where $\lambda_{2j+1}^{(s)}$ and $\beta_{2j+1}^{(k-s-1)}$ are the eigenvalues of $G_1$ and $G_2$, respectively. 
 	
 \end{proposition}
 \begin{proof}
 	It is straightforward from Lemma \ref{lemma: autovalores dados por Chebyshev}, item $(1)$, and Lemma \ref{lemma: prop polinomios de Chebyshev}, items $(2)$ and $(3)$.
 \end{proof}
\begin{lemma}\label{lemma: lambda_k mas lambda_1 positivo}Let $G(\Z_n,S)$ be a circulant graph with $n=2k$ and $S=\{1,2,\ldots,s,n-s,\ldots,n-1\}$. If $2\leq s \leq \frac{k-1}{2}$ and $k\geq 6$, then $\lambda_k^{(s)}+\lambda_1^{(s)}>0$. 
\end{lemma}
\begin{proof}
	Recall that, from formula \eqref{formula: autovalores en fc de chebysshev 2 clase}, items \eqref{lambdak} and  \eqref{lambdacreciente} of Lemma \ref{lemma: autovalores dados por Chebyshev}, we have that $\lambda_k^{(s)}\geq -2$,  and then
	\begin{align*}
	\lambda_k^{(s)}+\lambda_1^{(s)}&\geq-2+\lambda_1^{(2)}\\      
	& \geq U_2(X_1)+U_1(X_1)-1-2\\           
	&\geq 2X_1(2X_1+1)-1-3.
	\end{align*}                              
	Then, we obtain that
	\begin{align*}
	\lambda_k^{(s)}+\lambda_1^{(s)}&\geq 2\cos\left(\frac{\pi}{k}\right)\left(2\cos\left(\frac{\pi}{k}\right)+1\right)-4\\      
	& \geq 2\cos\left(\frac{\pi}{6}\right)\left(2\cos\left(\frac{\pi}{6}\right)+1\right)-4>0.\          
	\end{align*}                           
	
\end{proof}

Next,  we derive some new bounds for the eigenvalue $\lambda_1^{(s)}$.
 
 \begin{proposition} \label{cota primer autovalor}
 	Let $k\geq 6$, $n=2k$, $S=\{1,2,\ldots,s,n-s,\ldots,n-1\}$, 
 	and $2\leq s\leq \frac{k-1}{2}$, then it holds
 	$$
 	\frac{s}{k}(2k-2s-1)<\lambda_1^{(s)}<\frac{7s}{2k}\left(k-s-\frac{1}{2}\right).
 	$$
 \end{proposition}
 \begin{proof}
 	We proceed by induction on $s$ to prove  both bounds. We start with the lower bound.
 	
 	Observe that 
 	\begin{equation}\label{ec: formula lambda s=2}
 	\lambda_1^{(2)}=2\cos\left(\frac{\pi}{k}\right)+2\cos\left(\frac{2\pi}{k}\right)=4\cos^2\left(\frac{\pi}{k}\right)+2\cos\left(\frac{\pi}{k}\right)-2.
 	\end{equation}
 	We shall prove that for any $k\geq 6$ it holds that 
 	\begin{eqnarray}\label{obs13}
 	4\cos^2\left(\frac{\pi}{k}\right)+2\cos\left(\frac{\pi}{k}\right)-2>\frac{2}{k}(2k-5).
 	\end{eqnarray}
 	Let $f(k)=4\cos^2\left(\frac{\pi}{k}\right)+2\cos\left(\frac{\pi}{k}\right)+\frac{10}{k}-6$ for any $k\geq 6$. The purpose is to find a lower bound for $f$. 
 	Then 
 	\begin{eqnarray}
 	f'(k)&=&\left(8\pi\cos\left(\frac{\pi}{k}\right)\sin\left(\frac{\pi}{k}\right)+2\pi\sin\left(\frac{\pi}{k}\right)-10\right)\left(\frac{1}{k^2}\right)\nonumber \\
 	&=&\left(4\pi\sin\left(\frac{2\pi}{k}\right)+2\pi\sin\left(\frac{\pi}{k}\right)-10\right)\left(\frac{1}{k^2}\right).\nonumber \
 	\end{eqnarray}
 	It is clear that there  exists an unique $k_0\in [6, +\infty)$ such that $f'(k)\geq 0$ for any $k\in [6, k_0]$ and  $f'(k)<0$ for all $k>k_0$. 
 	On the other hand, it is evident that $f(6)>0$ and 
 	\begin{eqnarray*}
 	\lim\limits_{k\to +\infty} f(k)=0.
 	\end{eqnarray*}
 	Note that we have actually proved that $f(k)>0$ for any $k\geq 6$ which
 	establishes the inequality \eqref{obs13}. 
 	
 	Now we prove that  $\lambda_1^{(s)}>\frac{s}{k}(2k-2s-1).$ Suppose that $\lambda_1^{(s-1)}> \frac{(s-1)}{k}(2k-2(s-1)-1)$.
 	Observe that, by inductive hypothesis,  $\lambda_1^{(s)}=\lambda_1^{(s-1)}+2\cos\left(s\frac{\pi}{k}\right)>\frac{(s-1)}{k}(2k-2(s-1)-1)+2\cos\left(s\frac{\pi}{k}\right)$.
 	Then
 	\begin{align*}
 	\lambda_1^{(s)}&> \frac{s}{k}(2k-2s-1)+\frac{2s}{k}-\frac{1}{k}(2k-2s+1)+2\cos(s\frac{\pi}{k}),
 	\end{align*}
 	or equivalently
 	\begin{align*}
 	\lambda_1^{(s)}&> \frac{s}{k}(2k-2s-1)+\frac{4s}{k}-\frac{2k+1}{k}+2\cos\left(s\frac{\pi}{k}\right).
 	\end{align*}
 	It remains to prove that
 	\begin{equation*}\label{obs13-2-1}
 	\frac{4s}{k}-\frac{2k+1}{k}+2\cos\left(s\frac{\pi}{k}\right)\geq 0
 	\end{equation*}  

	For any $s$ fixed, with $s\geq 3$ and $k\geq 2s+1$, we consider the function $$f(k)=	\frac{4s}{k}-\frac{2k+1}{k}+2\cos\left(s\frac{\pi}{k}\right).$$
 	Then, $f'(k)=\frac{1}{k^2}\left(1-4s+2s\pi\sin\left(\frac{s\pi}{k}\right)\right)$=$\frac{1}{k^2}g(k)$  with $g'(k)=-2\frac{s^2\pi^2}{k^2}\cos\left(\frac{s\pi}{k}\right)<0$ for any $2s+1\leq k$. As, for any $s\geq 3$, it hold that $g(2s+1)\geq 1-4s+2s\pi \sin\left(\frac{3\pi}{7}\right)>0$ and $\lim\limits_{k\to +\infty} g(k)=1-4s$, then  there exists a unique $k_0\geq 2s+1$ such that $g(k_0)=0, g(k)>0$  if $k<k_0$ and $g(k)<0$ for any $k>k_0.$ Thus, we conclude that $f$ is an increasing function on $[2s+1, k_0)$ and a decreasing function on $(k_0, +\infty).$
 	
 		Now, we consider the function
 $$
 	h(s)=f(2s+1)=2\cos\left(\frac{s\pi}{2s+1}\right)-\frac{3}{2s+1}.
 $$Since $h'(s)<0$ and $\lim\limits_{s\to +\infty}h(s)=0,$ we conclude that $f(2s+1)>0$ for any $s\geq 3.$ Finally, as $\lim\limits_{k\to +\infty}f(k)=0$, we conclude that $f(k)>0$ for any $k\geq 2s+1$ which
 completes the proof.

 	Now, we prove that the upper bound holds. According to \eqref{ec: formula lambda s=2}, we have to prove that $4\cos^2(\frac{\pi}{k})+2\cos(\frac{\pi}{k})-2 < \frac{7}{k}(k-\frac{5}{2})$, but this inequality holds since $4\cos^2(\frac{\pi}{k})+2\cos(\frac{\pi}{k})-2\leq 4<7-\frac{35}{2k}.$
 	
 	Suppose that $\lambda_1^{(s-1)}<\frac{7(s-1)}{2k}(k-(s-1)-\frac{1}{2})$, by inductive hyphotesis. 
 	Observe that $\lambda_1^{(s)}=\lambda_1^{(s-1)}+2\cos(s\frac{\pi}{k})<\frac{7(s-1)}{2k}(k-(s-1)-\frac{1}{2})+2\cos(s\frac{\pi}{k})$
 	Then
 	\begin{align*}
 	\lambda_1^{(s)}& <\frac{7s}{2k}\big(k-s-\frac{1}{2}\big)+\frac{7s}{2k}-\frac{7}{2k}\big(k-s+\frac{1}{2}\big)+2\cos(s\frac{\pi}{k}).
 	\end{align*}
 	We have to prove that 
 	\begin{equation}\label{obs13-3-1}
 	\frac{7s}{2k}-\frac{7}{2k}\big(k-s+\frac{1}{2}\big)+2\cos(s\frac{\pi}{k})\leq 0.
 	\end{equation}
 	We remark  that \eqref{obs13-3-1} is equivalent to 
 	\begin{equation*}\label{obs13-3-equiv}
 	2\cos\left(\frac{s\pi}{k}\right)+\frac{7s}{k}-\frac{7}{4k}\leq \frac{7}{2}.
 	\end{equation*}
 	Let for any $k\geq 2s+1$ and $s$ fixed, $f_s(k)=2\cos\left(\frac{s\pi}{k}\right)+\frac{7s}{k}-\frac{7}{4k}.$ 
 	Then 
 	\begin{equation*}\label{obs13-3}
 	f_s'(k)=\left(2s\pi\sin\left(\frac{s\pi}{k}\right)-7s+\frac{7}{4}\right)\left(\frac{1}{k^2}\right)\leq  0, 
 	\end{equation*}
 	if $s\geq 3$. 
Therefore it is sufficient to show  that, for any $k\geq 2s+1$ we have 
 	\begin{equation*}
 f_{s}(2s+1)\leq \frac{7}{2}.
 	\end{equation*}
 	For the sake of clarity, we denote by $g(s)=f_{s}(2s+1)$. Since $g'(s)>0$, then 
 	$$
 	f_{s}(2s+1)=g(s)\leq \lim\limits_{s\to+\infty} g(s)=\frac{7}{2},
 	$$
 	which is the desired conclusion.   
 \end{proof}
\subsection{Generalization of Corollary~\ref{cor:distinctinertia}} 

 In what follows we shall study  particular families of circulant graphs with $n=2k$. Our aim is to prove that  pairs of families of  circulant graphs $G_1$ and $G_2$, defined as in Theorem~\ref{thm: same inertia}, are NCSC. 
  \begin{theorem}\label{lambdas}
For any $n=2k$, $k\ge 6$ and $2\le s\le k-3$, $G_1$ and $G_2$ are  \textbf{NCSC} circulant graphs.
\end{theorem}
 
  We begin by showing that the first eigenvalues of  both graphs are of distinct sign. This fact is the key to prove that 
 $G_1$ and  $G_2$ are not cospectral.

\begin{lemma} \label{lemma:propiedades de los autovalores} Let $\lambda_j^{(s)}$ and $\beta_j^{(s)}$, $0\leq j \leq n-1$, be the eigenvalues of $G_1$ and $G_2$ respectively. If	 $2\leq s\leq \frac{k-1}{2}$, then $\lambda_1^{(s)}>0$ and   $\beta_1^{(s)}<0$.
\end{lemma}

\begin{proof}

Note that by Lemma \ref{lemma: autovalores dados por Chebyshev} item \eqref{lambdaChebyshev} we get 
$$
\lambda_1^{(s)}=2\sum_{h=1}^s T_h(X_1)=2\sum_{h=1}^s \cos(h\frac{\pi}{k}).
$$ 
Since $\frac{\pi}{k}\leq \frac{h\pi}{k}<\frac{\pi}{2}$, then $\cos(h\frac{\pi}{k})>0$. Hence $\lambda_1^{(s)}>0$ and by Lemma \ref{lemma: autovalores dados por Chebyshev} item \eqref{beta igual lambda}, we conclude that $\beta_1^{(s)}<0$.

\end{proof}

From Lemma \ref{lemma: autovalores dados por Chebyshev}, items \eqref{beta igual lambda} and \eqref{simetria lambda}, we deduce the following result.

\begin{proposition}\label{prop: distinct sets}Let $n=2k$, $k\ge 6$ and  $2 \le s \le k-3$. If  the sets $$\left\{\lambda_1^{(s)},\lambda_3^{(s)},\ldots,\lambda_{2\lfloor\frac{k-1}{2}\rfloor +1}^{(s)}\right\} \   and \ \left\{\beta_1^{(s)},\beta_3^{(s)},\ldots,\beta_{2\lfloor\frac{k-1}{2}\rfloor +1} ^{(s)}\right\},$$ are distinct, then $G_1$ and $G_2$ are not cospectral. 

\end{proposition}

Next, we shall need the following statement. 

\begin{proposition}\label{cotainfU2s}
 Let $n=2k$, $Y_j=\cos(\frac{j\pi}{n})$ and $2\leq j \leq k-1$. Then  $U_{2s}(Y_j)\geq -\frac{s+1}{2}$ for any $s\geq 2.$
\end{proposition}
\begin{proof} From Lemma \ref{lemma: prop polinomios de Chebyshev} we have that
\begin{eqnarray}\label{derivadaU2s}
U'_{2s}(X)=\frac{(2s+1)T_{2s+1}(X)-XU_{2s}(X)}{X^2-1}.
\end{eqnarray}

If $z_l=\cos\left(l\frac{\pi}{2s+1}\right)$ then by \eqref{derivadaU2s} and Lemma \ref{lemma: prop polinomios de Chebyshev} we obtain 
\begin{eqnarray}\label{derivadaU2sz}
U'_{2s}(z_l)&=&\frac{(2s+1)T_{2s+1}(z_l)-z_lU_{2s}(z_l)}{z_l^2-1}\nonumber\\
&=&\frac{(2s+1)\cos\left((2s+1)l\frac{\pi}{2s+1}\right)-\cos\left(l\frac{\pi}{2s+1}\right)\frac{\sin\left((2s+1)l\frac{\pi}{2s+1}\right)}{\sin\left(l\frac{\pi}{2s+1}\right)}}{- \sin^2\left(l\frac{\pi}{2s+1}\right)}\nonumber\\
&=& \frac{-(2s+1)\cos(l\pi)}{\sin^2\left(l\frac{\pi}{2s+1}\right)}\left(1-\frac{1}{(2s+1)}\frac{\tan(l\pi)}{\tan\left(l\frac{\pi}{2s+1}\right)}\right)\nonumber \\
&=& \frac{-(2s+1)\cos(l\pi)}{\sin^2\left(l\frac{\pi}{2s+1}\right)}\left(1-\frac{\frac{\tan(l\pi)}{l\pi}}{\frac{\tan\left(l\frac{\pi}{2s+1}\right)}{l\frac{\pi}{2s+1}}} \right).
\end{eqnarray}

For a fixed $l_0\in ( 1, \frac 32)$ we have:

{\bf Claim 1:} The function $g_{l_0}(s)=\frac{-(2s+1)\cos(l_0\pi)}{\sin^2\left(l_0\frac{\pi}{2s+1}\right)}$ is positive if $s\geq 2.$ 

{\bf Claim 2:} If we denote by 
$$
F_{l_0}(s)=1-\frac{\frac{\tan(l_0\pi)}{l_0\pi}}{\frac{\tan\left(l_0\frac{\pi}{2s+1}\right)}{l_0\frac{\pi}{2s+1}}},
$$
we conlude that the sign of $U'_{2s}\left(\cos\left(l_0\frac{\pi}{2s+1}\right)\right)$ is determinated by the sign of $F_{l_0}(s)$. It is easy to prove that  $f(x)=\frac{\tan(x)}{x}$ is a strictly increasing function on $(0, \frac{\pi}{2})$ and so  $F
_{l_0}$ is a decreasing function.

%


We now consider three cases:
\begin{enumerate}
	\item Case 1: If $s=2.$ Let $l_{1, 2}=\frac{185}{128}, l_{2, 2}=\frac{186}{128}\in (1, \frac 32)$, then we have that 
	\begin{equation}
	U'_{4}\left(\cos\left(l_{1, 2}\frac{\pi}{5}\right)\right)>0\quad {\rm and}\quad U'_{4}\left(\cos\left(l_{2, 2}\frac{\pi}{5}\right)\right)<0.
	\end{equation}
	\item  Case 2:  Let $s=3,$ $l_{1, 3}=\frac{184}{128}$ and $l_{2, 3}=\frac{185}{128}\in (1, \frac 32)$, then we have that 
	\begin{equation}
	U'_{6}\left(\cos\left(l_{1, 3}\frac{\pi}{7}\right)\right)>0\quad {\rm and}\quad U'_{6}\left(\cos\left(l_{2, 3}\frac{\pi}{7}\right)\right)<0.
	\end{equation}
	\item Case 3: Finally if $s\geq 4$, we consider $l_{1, 4}=\frac{91}{64}, l_{2, 4}=\frac{92}{64}\in (1, \frac 32)$, by Claim 1 and 2, we 
conclude that 
	\begin{equation}
F_{l_{1, 4}}(4)>0,\: \:F_{l_{2, 4}}(4)<0\qquad {\rm and} \quad \lim\limits_{s\to +\infty} F_{l_{1, 4}}(s)>0.
	\end{equation}
	Then, for any $s\geq 4$ we establish that
	\begin{equation}
	U'_{2s}\left(\cos\left(l_{1,4}\frac{\pi}{2s+1}\right)\right)>0\quad {\rm and}\quad U'_{2s}\left(\cos\left(l_{2, 4}\frac{\pi}{2s+1}\right)\right)<0.
	\end{equation}
\end{enumerate}

From the above conclusion, we have that for any $s\geq 2$ there exists $l_0\in (l_{1, s}, l_{2, s})$   such that $U'_{2s}(z_{l_0})=0.$ Therefore
\begin{equation}
U_{2s}(z_{l_0})=\frac{(2s+1)T_{2s+1}(z_{l_0})}{z_{l_0}}.
\end{equation}

If we prove that 
\begin{equation}\label{ineq1}
\frac{(2s+1)T_{2s+1}(z_{l})}{z_{l}}+\frac{s+1}{2}>0,
\end{equation}
for any $l\in (l_{1, s}, l_{2, s})$, then we obtain 
\begin{equation*}
U_{2s}(z_{l_0})>\frac{-(s+1)}{2}.
\end{equation*}
The left side of \eqref{ineq1} is written as 
\begin{equation}
U_{2s}(z_l)+\frac{s+1}{2}=\frac{(4s+2)\cos(l\pi)+(s+1)\cos\left(\frac{l\pi}{2s+1}\right)}{2\cos\left(\frac{l\pi}{2s+1}\right)},
\end{equation}
because $2\cos\left(\frac{l\pi}{2s+1}\right)>0$, to obtain the inequality \eqref{ineq1} it is enough to show that
\begin{equation}
N(l)=(4s+2)\cos(l\pi)+(s+1)\cos\left(\frac{l\pi}{2s+1}\right)>0.
\end{equation}

The derivate of $N$ is given by
\begin{eqnarray}
N'(l)&=&-\pi(4s+2)\sin(l\pi)-\frac{(s+1)\pi}{2s+1}\sin\left(\frac{l\pi}{2s+1}\right)\nonumber\\
&=&\pi\left[(4s+2)(-\sin(l\pi))-\left(\frac{s+1}{2s+1}\right)\sin\left(\frac{l\pi}{2s+1}\right)\right].\nonumber\
\end{eqnarray}

Using the fact that if $s\geq 2$ and $l\in I=\left(\frac{91}{64}, \frac{93}{64} \right)$ then 
$$
N'(l)>\pi(10(-\sin(l\pi))-1)\geq 0.
$$
Thus, $N$ is a strictly increasing function on $I$ for any $s\geq 2.$ Then
\begin{enumerate}
	\item If $s=2$, as $N\left(\frac{185}{128}\right)>0$ then $N(l)>0$ in $(l_{1,2}, l_{2, 2})$ and by \eqref{ineq1} we conclude that $U_4(z_{l_0})>-\frac{3}{2}.$
	\item If $s=3$, as $N\left(\frac{184}{128}\right)>0$ then $N(l)>0$ in $(l_{1, 3}, l_{2, 3})$ and by \eqref{ineq1} we conclude that $U_6(z_{l_0})>-\frac{4}{2}.$
	\item If $s\geq 4$, as $N\left(\frac{91}{64}\right)>0$ then $N(l)>0$ in $(l_{1, 4}, l_{2, 4})$ and by \eqref{ineq1} we conclude that 
$U_{2s}(z_{l_0})>-\frac{s+1}{2}.$
\end{enumerate}
The extrema of $U_n(X)$ are not in general as readily determined. All that we can say for certain is that the absolute value of the extreme values of $U_n(X)$ increases monotonically as $|X|$ increases away from $0$. 
Therefore, the minimum of $U_{2s}$ with largest absolute value belongs to the interval $(z_2,1)$. Since $z_{l_0} \in \big(\cos(\frac{3}{2}\frac{\pi}{2s+1}),\cos(\frac{\pi}{2s+1})\big)\subset (z_2,1),$  we have that $z_{l_0}$ is the minimum of largest absolute value of $U_{2s}$. Then for all $2\leq j \leq k-1$, we obtain that $U_{2s}(Y_j)\geq U_{2s}(z_{l_0})\geq -\frac{s+1}{2}$.

\end{proof}

\begin{proof}[Proof of Theorem \ref{lambdas}]

	Let $n=2k$ and $2\le s \le \frac{k-1}{2}$. From   Propositions \ref{cota primer autovalor} and  \ref{cotainfU2s},   it is clear that $\lambda_1^{(s)}>\frac{s}{k}(2k-2s-1)$ and, combining with equation \eqref{eq: lambda dados por U} in proof of   Lemma \ref{lemma: autovalores dados por Chebyshev}, it follows immediately that
	$$
	\lambda_j^{(s)}>-\frac{s+1}{2}-1.
	$$
	Combining these inequalities we obtain that
	\begin{equation}\label{sumalambdas}
	\lambda_1^{(s)}+\lambda_j^{(s)}>\frac{s}{k}(2k-2s-1)-\frac{s+1}{2}-1\geq 0,
	\end{equation}
	for any $k\geq 7$, and $2\leq j \leq k-1.$

From Lemma \ref{lemma: lambda_k mas lambda_1 positivo} we have that 
$$
\lambda_1^{(s)}+\lambda_k^{(s)}> 0,
$$
for any $k\geq 6$.

Finally,  for the case $s=2$ and $k=6$ the inequality \eqref{sumalambdas} 	can be showed directly in using Sage.

Now, let $\frac{k-1}{2} < s \le k-3$ and $s'=k-s-1$. Since $2\le s' \le \frac{k-1}{2}$, from Proposition \ref{lemma: autovalores (s)=(k-s-1)}, we have that
$$
\lambda_1^{(s)}+\lambda_{2i+1}^{(s)}= \lambda_1^{(s')}+\lambda_{2i+1}^{(s')} > 0,
$$
for any $1 \le i \le \lfloor\frac{k-1}{2}\rfloor$.

In conclusion, we proved that 
\begin{equation}\label{desigualdad}
\lambda_1^{(s)}+\lambda_{2i+1}^{(s)}> 0,
\end{equation}
for any $1 \le i \le \lfloor\frac{k-1}{2}\rfloor$, and $2\le s \le k-3$.

From Lemma \ref{lemma: autovalores dados por Chebyshev} item \eqref{beta igual lambda}, we have that $\lambda_{2i+1}^{(s)} = - \beta_{2i+1}^{(s)}$ for any $0 \le i \le \lfloor\frac{k-1}{2}\rfloor$, and $2\le s \le k-3$. Combining this fact with inequality \eqref{desigualdad} and Lemma \ref{lemma:propiedades de los autovalores}, we show that $\lambda_1^{(s)}\neq \beta_{2i+1}^{(s)}$ with $0\leq i \leq \lfloor \frac{k-1}{2}\rfloor$. Therefore,  the sets 
$$
\left\{\lambda_1^{(s)},\lambda_3^{(s)},\ldots,\lambda_{2\lfloor\frac{k-1}{2}\rfloor +1}^{(s)}\right\} \ {\rm and} \ \left\{\beta_1^{(s)},\beta_3^{(s)},\ldots,\beta_{2\lfloor\frac{ k-1}{2}\rfloor +1} ^{(s)}\right\},
$$
are distinct. Finally, according to Proposition \ref{prop: distinct sets}, we conclude that $G_1$ and $G_2$ are not cospectral. 

\end{proof}

\section{Circulant graphs with an odd prime number of vertices}\label{sec: prime-case}
In this section we prove that for circulant graphs with an odd number of vertices, to be singularly cospectral implies to be isomorphic. We recall firstly
 the following  statement due to Turner.

\begin{theorem}\cite{ElpasandTurner1970}\label{thm: circulant matrices with the same eigenvalues}
	Let $A$ and $B$ be n $\times$ n circulant matrices with rational entries, where $n$ is an odd  prime integer number.  If $A$ and $B$ have the same eigenvalues, then $A$ and $B$ are permutationally similar. That is, there  exists a permutation matrix $P$ such that $B= P^{-1}AP$. 
	Moreover, $P$ can be selected in the special form $P=(P_{qi,j})$ where $q$ is an integer $1\le q \le n- 1$, and $P_{qi,j}= 1$ if $qi\equiv j \ (\textrm{mod } p)$ and $P_{qi,j}=0$ otherwise.
\end{theorem}

\begin{remark}\label{rmk: preserv adjacency matrix}
	If $A$ is a circulant adjacency matrix of a circulant graph $G$ of  prime order, and $P$ is defined as in Theorem~\ref{thm: circulant matrices with the same eigenvalues}, then $P^{-1}AP$ is a circulant adjacency matrix of $G$.
	\end{remark}
%
\begin{lemma}~\label{lem: square of the adjacency matrix}
	If $H$ is a circulant graph with  a positive integer $n$, then $H$ is isomorphic to the circulant graph $G(\Z_n,S_H)$, where $S_H=\{x:2x=j\mbox{ and }\,(A_H^2)_{0, j}\mbox{ is  odd}\}$ (sums should be considered modulo $n$) and $A_H$ is a circulant adjacency matrix of $H$. 
\end{lemma}
\begin{proof}
	Let $H$ be a circulant graph and let $v_0,\ldots,v_{n-1}$ an order of its vertices such that the adjacency matrix $A_H$ of $H$ indexed by such an order is a circulant matrix. Let $S$ be the set of indices $j$ such that $(A_H)_{0,j}=1$, meaning that $H$ is isomorphic to $G(\Z_n,S)$. Recall that $(A_H^2)_{i,j}$ equals the number of walks of length two between $v_i$ and $v_j$ for every $0\le i,j\le n-1$. Notice that $(A_H^2)_{0,j}$ is equal to the number of ordered pairs $(x,y)\in S\times S$ such that $x+y=j$, as $H$ is a circulant graph, and where the sums should be considered modulo $n$. Thus $(A_H^2)_{0,j}$ is even, unless $2x=j$ for some $x\in S$ ($v_x$ is adjacent to $v_0$). Consequently, given $A_H^2$, the set $S_H$ of neighbors of $v_0$ are precisely those $v_x$ such that $2x=j$ and $(A_H^2)_{0,j}$ is odd. Therefore, $G(\Z_n,S_H)$ is isomorphic to $H$.
\end{proof}
We can now prove the main result of this section.
\begin{theorem}
	Let $F$ and $H$ be circulant graphs both with an odd prime number of vertices. If $F$ and $H$ are singularly cospectral, then $F$ and $H$ are isomorphic.
\end{theorem}

\begin{proof}
	Let $A_H$ and $A_F$ circulant adjacency matrices of the circulant graphs $F$ and $H$. If $F$ and $H$ are singularly cospectral circulant graphs, then the circulant matrices $A_F^2$ and $A_H^2$ have the same eigenvalues. It is well-known that by multiplying two circulant matrices is obtained a circulant matrix. Thus,  $A_F^2=P^{-1}A_H^2P=(P^{-1}A_HP)^2$ for a permutation matrix $P$ defined as in Theorem~\ref{thm: circulant matrices with the same eigenvalues}. Besides, $\hat{A}_H:=P^{-1}A_HP$ is a circulant adjacency matrix for $H$ (see Remark~\ref{rmk: preserv adjacency matrix}). Consequently, if $S_F$ and $S_H$ are defined as in the proof of Lemma~\ref{lem: square of the adjacency matrix}, for $A_F$ and $\hat{A}_H$, then $F$ and $H$ are isomorphic to $G(\Z_p,S_F)$ and $G(\Z_p,S_H)$, respectively. Therefore, by Lemma~\ref{lem: square of the adjacency matrix}, $F$ and $H$ are isomorphic, as $A_F^2=\hat{A}_H^2$ and thus $S_F=S_H$.
\end{proof}

\bibliographystyle{abbrv}
\bibliography{bibliografia_cdgp}


\end{document}